\documentclass[12pt]{proc-l}

\usepackage{amsmath}
\usepackage{amsthm}
\usepackage{amssymb}
\usepackage{amsfonts}
\usepackage{fullpage}

\newtheorem{theorem}{Theorem}[section]

\newtheorem{lemma}[theorem]{Lemma}

\newtheorem{proposition}[theorem]{Proposition}

\numberwithin{equation}{section}

\def\C{{\mathbb C}}

\def\N{{\mathbb N}}

\def\R{{\mathbb R}}

\def\E{{\mathbb E}}

\def\eps{\varepsilon}

\newcommand{\abs}[1]{\left\lvert#1\right\rvert}

\begin{document}

\title[Real zeros of random orthogonal polynomials]{Expected number of real zeros for random linear combinations of orthogonal polynomials}

\author{D. S. Lubinsky}
\thanks{Research of D. S. Lubinsky was partially supported by NSF grant DMS136208 and US-Israel BSF grant 2008399.}
\address{School of Mathematics, Georgia Institute of Technology, Atlanta, GA 30332, USA}
\email{lubinsky@math.gatech.edu}

\author{I. E. Pritsker}
\thanks{Research of I. E. Pritsker was partially supported by the National Security Agency (grant H98230-12-1-0227) and by the AT\&T Foundation.}
\address{Department of Mathematics, Oklahoma State University, Stilwater, OK 74078, USA}
\email{igor@math.okstate.edu}

\author{X. Xie}
\address{Department of Mathematics, Oklahoma State University, Stilwater, OK 74078, USA}
\email{sophia.xie@okstate.edu}


\date{}

\commby{Walter Van Assche}


\begin{abstract}
We study the expected number of real zeros for random linear combinations of orthogonal polynomials. It is well known that Kac polynomials, spanned by monomials with i.i.d. Gaussian coefficients, have only  $(2/\pi + o(1))\log{n}$ expected real zeros in terms of the degree $n$. On the other hand, if the basis is given by Legendre (or more generally by Jacobi) polynomials, then random linear combinations have $n/\sqrt{3} + o(n)$ expected real zeros. We prove that the latter asymptotic relation holds universally for a large class of random orthogonal polynomials on the real line, and also give more general local results on the expected number of real zeros.
\end{abstract}

\maketitle

\section{Background}

Zeros of polynomials with random coefficients have been intensively studied since 1930s. The early work concentrated on the expected number of real zeros $\E[N_n(\R)]$ for polynomials of the form $P_n(x)=\sum_{k=0}^{n} c_k x^k,$ where $\{c_k\}_{k=0}^n$ are independent and identically distributed random variables. Apparently the first paper that initiated the study is due to Bloch and P\'olya \cite{BP}. They gave an upper bound $\E[N_n(\R)] = O(\sqrt{n})$ for polynomials with coefficients selected from the set $\{-1,0,1\}$ with equal probabilities. Further results generalizing and improving that estimate were obtained by Littlewood and Offord \cite{LO1}-\cite{LO2}, Erd\H{o}s and Offord \cite{EO} and others. Kac \cite{Ka1} established the important asymptotic result
\[
\E[N_n(\R)] = (2/\pi + o(1))\log{n}\quad\mbox{as }n\to\infty,
\]
for polynomials with independent real Gaussian coefficients. More precise forms of this asymptotic were obtained by many authors, including Kac \cite{Ka2}, Wang \cite{Wa}, Edelman and Kostlan \cite{EK}. It appears that the sharpest known version is given by the asymptotic series of Wilkins \cite{Wi1}. Many additional references and further directions of work on the expected number of real zeros may be found in the books of Bharucha-Reid and Sambandham \cite{BRS}, and of Farahmand \cite{Fa}. In fact, Kac \cite{Ka1}-\cite{Ka2} found the exact formula for $\E[N_n(\R)]$ in the case of standard real Gaussian coefficients:
\[
\E[N_n(\R)]=\frac{4}{\pi}\displaystyle \int_0^1 \frac{\sqrt{A(x)C(x)-B^2(x)}}{A(x)} \, dx,
\]
where
\[
A(x)=\sum_{j=0}^n x^{2j},\quad
B(x)=\sum_{j=1}^n j x^{2j-1}\quad\mbox{and}\quad
C(x)=\sum_{j=1}^n j^2 x^{2j-2}.
\]
In the subsequent paper Kac \cite{Ka3}, the asymptotic result for the number of real zeros was extended to the case of uniformly distributed coefficients on $[-1, 1]$. Erd\H{o}s and Offord \cite{EO} generalized the Kac asymptotic to Bernoulli distribution (uniform on $\{-1,1\}$), while Stevens \cite{St} considered a wide class of distributions. Finally, Ibragimov and Maslova \cite{IM1,IM2} extended the result to all mean-zero distributions in the domain of attraction of the normal law.

We state a result on the number of real zeros for the random linear combinations of rather general functions. It originated in the papers of Kac \cite{Ka1}-\cite{Ka3}, who used the monomial basis, and was extended to trigonometric polynomials and other bases, see Farahmand \cite{Fa} and Das \cite{Da1}-\cite{Da2}. We are particularly interested in the bases of orthonormal polynomials, which is the case considered by Das \cite{Da1}. For any set $E\subset\C$, we use the notation $N_n(E)$ for the number of zeros of random functions \eqref{1.1} (or random orthogonal polynomials of degree at most $n$) located in $E$. The expected number of zeros in $E$ is denoted by $\E[N_n(E)],$ with $\E[N_n(a,b)]$ being the expected number of zeros in $(a,b)\subset\R.$

\begin{proposition} \label{prop1.1}
Let $[a,b]\subset\R$, and consider real valued functions $g_{j}(x)\in C^1([a,b]), \ j=0,\ldots,n,$ with $g_{0}(x)$ being a nonzero constant. Define the random function
\begin{equation} \label{1.1}
G_n(x)=\sum_{j=0}^{n} c_{j}g_{j}(x),
\end{equation}
where the coefficients $c_j$ are i.i.d. random variables with Gaussian distribution $\mathcal{N}(0, \sigma^2), \sigma > 0$. If there is $M\in\N$ such that $G_n'(x)$ has at most $M$ zeros in $(a,b)$ for all choices of coefficients, then the expected number of real zeros of $G_n(x)$ in the interval $(a,b)$ is given by
\begin{equation} \label{1.2}
\E[N_n(a,b)]=\frac{1}{\pi} \int_a^b \frac{\sqrt{A(x)C(x)-B^2(x)}}{A(x)} \, dx
\end{equation}
where
\begin{align} \label{1.3}
A(x)=\sum_{j=0}^n g_j^{2}(x), \quad B(x)=\sum_{j=1}^n g_j(x)g_j'(x) \quad \mbox{and}\quad C(x)=\sum_{j=1}^n [g_j'(x)]^2.
\end{align}
\end{proposition}

Clearly, the original formula of Kac follows from this proposition for $g_j(x)=x^j,\ j=0,1,\ldots,n.$ We sketch a proof of Proposition \ref{prop1.1} in Section 3, as we could not find a suitable reference with a complete proof for Proposition \ref{prop1.1} in this general form. We note that multiple zeros are counted only once by the standard convention in all of the above results on real zeros. However, the probability of having a multiple zero for a polynomial with Gaussian coefficients is equal to 0, so that we have the same result on the expected number of zeros regardless whether they are counted with or without multiplicities.

\section{Random orthogonal polynomials}
Let $\mu$ denote a positive Borel measure compactly supported on the real line, with infinitely many points in its support, and with finite power moments of all orders. For $n\geq 0$, let
\begin{equation*}
p_{n}\left( x \right) =\gamma _{n}x^{n}+...
\end{equation*}%
denote the $n$th orthonormal polynomial for $\mu $, with $\gamma_n>0,$ so that
\begin{equation*}
\int p_{n}p_{m}d\mu =\delta _{mn}.
\end{equation*}%
Using the orthonormal polynomials $\{p_j\}_{j=0}^{\infty}$ as the basis, we consider the ensemble of random polynomials of the form
\begin{equation} \label{2.1}
P_n(x)=\sum_{j=0}^{n}c_{j}p_{j}(x),\quad n\in\N,
\end{equation}
where the coefficients $c_0,c_1,\ldots,c_n$ are i.i.d. random variables. Such a family is often called random orthogonal polynomials. If the coefficients have Gaussian distribution, one can apply Proposition \ref{prop1.1} to study the expected number of real zeros of random orthogonal polynomials. In particular, Das \cite{Da1} considered random Legendre polynomials, and found that $\E[N_n(-1,1)]$ is asymptotically equal to $n/\sqrt{3}$. Wilkins \cite{Wi2} improved the error term in this asymptotic relation by showing that $\E[N_n(-1,1)] = n/\sqrt{3} + o(n^\eps)$ for any $\eps>0.$ For random Jacobi polynomials, Das and Bhatt \cite{DB} concluded that $\E[N_n(-1,1)]$ is asymptotically equal to $n/\sqrt{3}$ too. They also provided estimates for the expected number of real zeros of random Hermite and Laguerre polynomials, but those arguments contain significant gaps. Farahmand \cite{Fa,Fa1,Fa2} considered various generalizations of these results for the level crossings of random sums of Legendre polynomials with coefficients that may have different distributions. Interesting computations and pictures of zeros of random orthogonal polynomials may be found on the \textsc{chebfun} web page of Trefethen \cite{Tr}.

For the orthonormal polynomials $\{p_j(x)\}_{j=0}^{\infty}$ associated with the measure $\mu$, define the reproducing kernel by
\[
K_n(x,y)=\sum_{j=0}^{n-1}p_j(x)p_j(y),
\]
and the differentiated kernels by
\[
K_n^{(k,l)}(x,y)=\sum_{j=0}^{n-1}p_j^{(k)}(x)p_j^{(l)}(y),\quad k,l\in\N\cup\{0\}.
\]
The strategy is to apply Proposition \ref{prop1.1} with $g_j=p_j$, so that
\begin{align} \label{2.2}
A(x) = K_{n+1}(x,x), \quad B(x) = K_{n+1}^{(0,1)}(x,x) \quad \mbox{and}\quad C(x) = K_{n+1}^{(1,1)}(x,x).
\end{align}
We use universality limits for the reproducing kernels of orthogonal polynomials (see Lubinsky \cite{Lu1}-\cite{Lu2} and Totik \cite{To1}-\cite{To2}), and asymptotic results on zeros of random polynomials (cf. Pritsker \cite{Pr}) to give asymptotics for the expected number of real zeros for a wide class of random orthogonal polynomials.

\begin{theorem} \label{thm2.1}
Let $K\subset\R$ be a finite union of closed and bounded intervals, and let $\mu$ be a positive Borel measure supported on $K$ such that $d\mu(x)=w(x)dx$ and $w>0$ a.e. on $K$. If for every $\eps>0$ there is a closed set $S\subset K$ of Lebesgue measure $|S|<\eps$, and a constant $C>1$ such that $C^{-1} < w < C$ a.e. on $K\setminus S,$ then the expected number of real zeros of random orthogonal polynomials \eqref{2.1} with Gaussian coefficients satisfy
\begin{align} \label{2.3}
\lim_{n\to\infty} \frac{1}{n} \E[N_n(\R)]= \frac{1}{\sqrt{3}}.
\end{align}
\end{theorem}
A simple example of the orthogonality measure $\mu$ satisfying the above conditions is given by the  density $w$ that is continuous on $K$ except for finitely many points, and has finitely many zeros on $K$. More specifically, one may consider the generalized Jacobi weight of the form $w(x)=v(x)\prod_{j=1}^J |x-x_j|^{\alpha_j},$ where $v(x) > 0,\ x\in K,$ and $\alpha_j>-1,\ j=1,\ldots,J.$

Theorem \ref{thm2.1} is a consequence of more precise and general local results given below. In order to state them, we need the notion of the equilibrium measure $\nu _{K}$ of a compact set $K\subset\C$. This is the unique probability measure supported on $K$ that minimizes the energy
\begin{equation*}
I[\nu] = - \iint \log |z-t|\, d\nu(t) d\nu(z)
\end{equation*}
amongst all probability measures $\nu$ with support on $K$. The logarithmic capacity of $K$ is
\begin{equation*}
\textup{cap}(K) = \exp\left(-I[\nu _{K}]\right) .
\end{equation*}
When we say that a compact set $K$ is regular, this means regularity in the sense of Dirichlet problem (or potential theory). See Ransford \cite{Ra} for further orientation.

We also need the notion of a measure $\mu$ regular in the sense of Stahl, Totik, and Ullman \cite{ST}. If $K=\textup{supp}\,\mu$ and
\begin{equation*}
\lim_{n\rightarrow \infty} \gamma _{n}^{1/n}=\frac{1}{\textup{cap}(K)},
\end{equation*}
where $\gamma_n$ is the leading coefficient of $p_n,$ then we say that $\mu$ is STU-regular. A sufficient condition for this is that $K$ consists of finitely many intervals and $\mu ^{\prime }=w>0$ a.e. in those intervals.

\begin{theorem} \label{thm2.2}
Let $\mu $ be an STU regular measure with compact support $K\subset\R$, which is regular in the sense of potential theory. Let $O$ be an open set in which $\mu$ is absolutely continuous, and such that for some $C>1$
\begin{equation} \label{2.4}
C^{-1}\leq \mu ^{\prime }\leq C\textup{ a.e. in }O.
\end{equation}
Then given any compact subinterval $[a,b]$ of $O$, we have
\begin{equation} \label{2.5}
\lim_{n\rightarrow\infty} \frac{1}{n}\E\left[N_n\left([a,b]\right)\right] = \frac{1}{\sqrt{3}} \nu_{K}([a,b]),
\end{equation}
where $\nu_K$ is the equilibrium measure of $K$.
\end{theorem}

This is a special case of the following result, where $\mu$ does not need
to be STU regular. The asymptotic lower bound requires very little of $\mu$.

\begin{theorem} \label{thm2.3}
Let $\mu$ be a measure on the real line with compact support $K$.\\
(a) Assume that $\mu^{\prime}>0$ a.e. in the interval $[a,b]$. Then
\begin{equation} \label{2.6}
\liminf_{n\rightarrow\infty} \frac{1}{n} \E\left[N_n\left([a,b]\right)\right] \geq \frac{1}{\sqrt{3}}\,\nu_{K}([a,b]).
\end{equation}
(b) Suppose in addition that \eqref{2.4} holds, and that $[a,b]\subset O$. Then
\begin{equation} \label{2.7}
\limsup_{n\rightarrow\infty} \frac{1}{n} \E\left[N_n\left([a,b]\right)\right] \leq \frac{1}{\sqrt{3}} \inf_L \,\nu_{L}([a,b]),
\end{equation}
where the $\inf$ is taken over all regular compact sets $L\subset K$ such that $L\supset[a,b]$, and the restriction $\mu\vert_L$ of $\mu$ to $L$ is STU regular.
\end{theorem}

It is plausible that the right hand sides of \eqref{2.6} and \eqref{2.7} are equal under mild assumptions such as the one of part (a). An interesting open problem is to find rates of convergence in the limit relations \eqref{2.3} and \eqref{2.5}.

\section{Proofs}
\begin{proof}[Proof of Proposition \ref{prop1.1}]
This proof is based on the discussions of Kac \cite[p. 5-10]{Ka1} and Das \cite{Da2}. The joint probability density of $\textbf{c}=(c_0,c_1,\cdots,c_n)$ is
\[
dP(\textbf{c}) = (2\pi)^{-(n+1)/2}\sigma^{-(n+1)}\displaystyle e^{-\frac{\|\textbf{c}\|^2}{2\sigma^2}}\,dc_0dc_1\cdots dc_n,
\]
where $\|\textbf{c}\|^2=c_0^2+c_1^2+\cdots+c_n^2$. Since $G_n(x)$ has at most $M+1$ zeros in $(a,b)$ for all $\textbf{c}$ by Rolle's theorem, $N_n(a,b)$ is integrable over $\mathbb R^{n+1}$ with respect to $dP(\textbf{c})$. Define
\[
N_n^{*}(a,b)=N_n(a,b)-(\kappa(a)+\kappa(b))/2,
\]
where
\begin{align*}
\kappa(x)=
\begin{cases}
\hfill 1 \hfill & \text{ if }G_n(x)=0,\\
\hfill 0 \hfill & \text{ otherwise}.\\
\end{cases}
\end{align*}
Since $G_n(a)$ and $G_n(b)$ are continuous random variables, we have
\[
\E[N_n(a,b)] = \displaystyle \int_{\mathbb R^{n+1}}^{} N_n^{*}(a,b) \, dP(\textbf{c})   . \]
We state the following result from Kac \cite[Theorem 1]{Ka2}.
\begin{lemma} \label{lem3.1}
If $f(x)$ is continuous for $\alpha \leq x \leq \beta$ and continuously differentiable for $\alpha < x < \beta$, and $f'(x)$ vanishes only at a finite number of points in $\alpha < x < \beta$, then the number of zeros of $f(x)$ in $\alpha < x < \beta$ (multiple zeros are counted once and if either $\alpha$ or $\beta$ is a zero, it is counted as $1/2$) is equal to
\[ 
\textup{P.V. } \frac{1}{2\pi} \displaystyle \int_{-\infty}^{\infty}  \displaystyle \int_{\alpha}^{\beta} \cos(y f(x))\abs{f'(x)} \,dx\,dy.
\]
\end{lemma}

In our notation, this gives
\[
N_n^{*}(a,b) = \textup{P.V. } \frac{1}{2\pi} \int_{-\infty}^{\infty} \displaystyle \int_{a}^{b} \cos(y G_n(x)) \abs{G_n'(x)} \,dx\,dy.
\]
Thus
\begin{align}
\nonumber
\E[N_n(a,b)] &= (2\pi)^{-\frac{n+1}{2}}\sigma^{-(n+1)} \displaystyle \int_{-\infty}^{\infty} \cdots \displaystyle \int_{-\infty}^{\infty} N_n^{*}(a,b)\displaystyle  e^{-\frac{\|\textbf{c}\|^2}{2\sigma^2}}\, dc_0dc_1\cdots dc_n\\
\label{3.1}
&=\frac{\sigma^{-(n+1)}}{2\pi} \int_{a}^{b} \int_{-\infty}^{\infty} R_n(x, y) \, dy \, dx,
\end{align}
where
\begin{equation}
\label{3.2}
R_n(x, y)=(2\pi)^{-\frac{n+1}{2}} \int_{-\infty}^{\infty} \cdots \int_{-\infty}^{\infty} e^{-\frac{\|\textbf{c}\|^2}{2\sigma^2}}\cos(y G_n(x))\abs{G_n'(x)}\, dc_0dc_1\cdots dc_n.
\end{equation}
The interchange of the integration order is justified by the fact that the integrand is dominated by
\[
 e^{-\frac{\|\textbf{c}\|^2}{2\sigma^2}} \sum_{j=0}^{n}\abs{c_j}\abs{g_j'(x)},
\]
which is exponentially small outside bounded sets in $\mathbb R^{n+1}$. We use the known relation
\begin{align} \label{3.3}
\frac{1}{\pi} \int_{-\infty}^{\infty} \frac{1-\cos(u v)}{u^2} \, du = \abs{v}
\end{align}
to write (\ref{3.2}) as
\begin{align} \label{3.4}
\nonumber
R_n(x, y) =&\ \textup{P.V. } \frac{1}{\pi} \int_{-\infty}^{\infty}\frac{du}{u^2}\ \times \\ (2\pi)^{-\frac{n+1}{2}} &\int_{-\infty}^{\infty} \cdots \int_{-\infty}^{\infty} e^{-\frac{\|\textbf{c}\|^2}{2\sigma^2}} \left(\cos(y G_n(x))-\cos(y G_n(x))\cos(u G_n'(x))\right)\,dc_0dc_1\cdots dc_n,
\end{align}
where the interchange of orders of the integration can be justified as above, and (\ref{3.4}) is interpreted as
\begin{equation} \label{3.5}
\lim_{N \rightarrow \infty} \lim_{\epsilon \rightarrow 0} \frac{1}{\pi} \left(\int_{-N}^{-\epsilon}+\int_{\epsilon}^{N}\right) (\cdots) \frac{du}{u^2}.
\end{equation}
Noting that
\begin{align*}
\cos(y G_n(x))\cos(u G_n'(x))=\frac{1}{2}\mathfrak{R}\left(e^{iy G_n(x)+iu G_n'(x)} + e^{iy G_n(x) - iu G_n'(x)}\right),
\end{align*}
we obtain with help of \cite[3.323(2) on p. 337]{GR} that
\begin{align*}
&(2\pi)^{-\frac{n+1}{2}}\displaystyle \int_{-\infty}^{\infty} \cdots \displaystyle \int_{-\infty}^{\infty} \displaystyle  e^{-\frac{\|\textbf{c}\|^2}{2\sigma^2}}\cos(y G_n(x))\cos(u G_n'(x))\,dc_0\cdots dc_n\\
&=\frac{(2\pi)^{-\frac{n+1}{2}}}{2}\mathfrak{R}\displaystyle \int_{-\infty}^{\infty} \cdots \displaystyle \int_{-\infty}^{\infty} \displaystyle  e^{-\frac{\|\textbf{c}\|^2}{2\sigma^2}} \left( \displaystyle e^{i\sum\limits_{j=0}^{n}[y c_jg_j(x)+u c_jg_j'(x)]}+ \displaystyle e^{i\sum\limits_{j=0}^{n}[y c_jg_j(x)-u c_jg_j'(x)]}  \right)\, dc_0\cdots dc_n\\
&=\frac{(2\pi)^{-\frac{n+1}{2}}}{2}\mathfrak{R}\left(  \prod\limits_{j=0}^{n} \displaystyle \int_{-\infty}^{\infty} \displaystyle e^{-\frac{c_j^2}{2\sigma^2}+i[y g_j(x)+u g_j'(x)]c_j} \, dc_j +\prod\limits_{j=0}^{n} \displaystyle \int_{-\infty}^{\infty} \displaystyle e^{-\frac{c_j^2}{2\sigma^2}+i[y g_j(x)-u g_j'(x)]c_j}\, dc_j \right)\\
&=\frac{(2\pi)^{-\frac{n+1}{2}}}{2}\mathfrak{R}\left( \prod\limits_{j=0}^{n}(2\pi)^{\frac{1}{2}}\sigma \displaystyle e^{-\frac{1}{2}[y g_j(x)+u g_j'(x)]^2\sigma^2}+\prod\limits_{j=0}^{n}(2\pi)^{\frac{1}{2}}\sigma \displaystyle e^{-\frac{1}{2}[y g_j(x)-u g_j'(x)]^2\sigma^2} \right)\\
&=\frac{\sigma^{n+1}}{2} \displaystyle e^{-\frac{\sigma^2}{2}\sum\limits_{j=0}^{n}[y g_j(x)+u g_j'(x)]^2}+\frac{\sigma^{n+1}}{2}\displaystyle e^{-\frac{\sigma^2}{2}\sum\limits_{j=0}^{n}[y g_j(x)-u g_j'(x)]^2}.
\end{align*}
For $u=0$, we have
\[
(2\pi)^{-\frac{n+1}{2}}\displaystyle \int_{-\infty}^{\infty} \cdots \displaystyle \int_{-\infty}^{\infty} \displaystyle e^{-\frac{\|\textbf{c}\|^2}{2\sigma^2}}\cos(y G_n(x)) \, dc_0\cdots dc_n = \sigma^{n+1} \displaystyle e^{-\frac{\sigma^2}{2}\sum\limits_{j=0}^{n}[y g_j(x)]^2}.
\]
Using abbreviations $A=A(x),\ B=B(x)$ and $C=C(x)$, we rewrite
\begin{align*}
R_n(x, y) &= \frac{\sigma^{n+1}}{\pi} \int_{-\infty}^{\infty} \frac{\displaystyle e^{-\frac{\sigma^2}{2}Ay^2}}{u^2}\,du \\ & - \frac{\sigma^{n+1}}{2\pi} \left( \int_{-\infty}^{\infty} \frac{e^{-\frac{\sigma^2}{2}(Ay^2+Cu^2+2y u B)}}{u^2}\,du + \int_{-\infty}^{\infty} \frac{e^{-\frac{\sigma^2}{2}(Ay^2+Cu^2-2y u B)}}{u^2}\,du \right)\\
&=\frac{\sigma^{n+1}}{\pi} e^{-\frac{\sigma^2}{2}Ay^2}  \int_{-\infty}^{\infty}\frac{1-e^{-\frac{\sigma^2}{2}Cu^2+y u B\sigma^2}} {u^2}\, du,
\end{align*}
where the integral exists as a principal value, in the sense indicated in (\ref{3.5}). If $C(x)=0$ for some $x$ then $B(x)=0$ and $R(x,y)=0$ for the same $x$ and all $y$. Thus we set $By\sigma^2=t$ and $ \sigma^2C=h>0$, so that
\[
R_n(x, y)=\frac{\sigma^{n+1}}{\pi} \displaystyle e^{-\frac{\sigma^2}{2}Ay^2} \displaystyle \int_{-\infty}^{\infty}\frac{1- \displaystyle e^{-\frac{1}{2}hu^2+tu}} {u^2}\, du.
\]
The  Taylor expansion
\[
e^{tu}=1+tu+\sum\limits_{m=2}^{\infty}\frac{t^mu^m} {m!},
\]
together with known identities
\[
\displaystyle \int_{-\infty}^{\infty}\frac{1- \displaystyle e^{-\frac{1}{2}hu^2}} {u^2}\,du = \sqrt{2\pi h}\quad \text{and}\quad
\textup{P.V.} \int_{-\infty}^{\infty}\frac{tu}{u^2}\, e^{-\frac{1}{2}hu^2}\,du=0,
\]
gives that
\[
R_n(x, y)=\frac{\sigma^{n+1}}{\pi} \displaystyle e^{-\frac{\sigma^2}{2}Ay^2}\left(\sqrt{2\pi h}-\displaystyle \int_{-\infty}^{\infty}\left(\sum\limits_{m=2}^{\infty}\frac{t^m u^m} {m!}\right)\displaystyle \frac{e^{-\frac{1}{2}hu^2}}{u^2}\,du\right).
\]
Assuming that $h>0$, we further obtain that
\begin{align*}
&\displaystyle \int_{-\infty}^{\infty}\left(\sum\limits_{m=2}^{\infty}\frac{t^m u^m} {m!}\right)\displaystyle \frac{e^{-\frac{1}{2}h u^2}}{u^2}\, du\\
&= \sum\limits_{m=1}^{\infty}\frac{t^{2m}}{(2m)!}\displaystyle \int_{-\infty}^{\infty}u^{2(m-1)}e^{-\frac{1}{2}h u^2}\,du\\
&= \sum\limits_{m=1}^{\infty}\frac{t^{2m}} {(2m)!}\displaystyle \frac{(2(m-1))!}{2^{m-1}(m-1)!}\sqrt{2\pi}h^{-(m-1)-\frac{1}{2}}\quad (\text{by \cite[3.461(2) on p. 364]{GR}})\\
&= \sum\limits_{m=1}^{\infty}\frac{\sqrt{2\pi h}}{m!(2m-1)}\left(\frac{t^2}{2h}\right)^m.
\end{align*}
Hence
\begin{align*}
R_n(x, y)&=\frac{\sigma^{n+1}}{\pi} \displaystyle e^{-\frac{\sigma^2}{2}Ay^2}\left(\sqrt{2\pi h}-\displaystyle \sum\limits_{m=1}^{\infty}\frac{\sqrt{2\pi h}}{m!(2m-1)}\left(\frac{t^2}{2h}\right)^m\right)\\
&=\sqrt{\frac{2C}{\pi}}\ \sigma^{n+2}\displaystyle e^{-\frac{\sigma^2}{2}Ay^2}\left(1-\displaystyle \sum\limits_{m=1}^{\infty}\frac{1}{m!(2m-1)}\left(\frac{B^2\sigma^2}{2C}\right)^my^{2m}\right).
\end{align*}
Applying \cite[3.461(2) on p. 364]{GR} again, we obtain that
\begin{align*}
\displaystyle \int_{-\infty}^{\infty}  R_n(x, y)\, dy&=\sqrt{\frac{2C}{\pi}}\ \sigma^{n+2}\displaystyle \int_{-\infty}^{\infty}\displaystyle e^{-\frac{\sigma^2}{2}Ay^2}\left(1-\displaystyle \sum\limits_{m=1}^{\infty}\frac{1}{m!(2m-1)}\left(\frac{B^2\sigma^2}{2C}\right)^my^{2m}\right) \, dy\\
&=\sqrt{\frac{2C}{\pi}}\ \sigma^{n+2} \left( \sqrt{\frac{2\pi}{A\sigma^2}}-\displaystyle \sum\limits_{m=1}^{\infty}\frac{(\frac{B^2\sigma^2}{2C})^m}{m!(2m-1)}\frac{(2m)!}{m!2^m}\sqrt{\frac{2\pi}{A\sigma^2}}\frac{1}{(A\sigma^2)^m} \right)\\
&=2\sqrt{\frac{C}{A}}\ \sigma^{n+1}\displaystyle\left( - \sum\limits_{m=0}^{\infty}\left(\frac{B^2}{AC}\right)^m\frac{(2m)!}{(m!)^2(2m-1)4^m}\right)\\
&=2\sqrt{\frac{C}{A}}\sqrt{1-\frac{B^2}{AC}}\ \sigma^{n+1}.
\end{align*}
Then (\ref{3.1}) gives us the desired formula
\[
\E[N_n(a,b)]=\frac{1}{\pi} \displaystyle \int_{a}^{b}\frac{\sqrt{AC-B^2}}{A}\, dx,
\]
where $AC-B^2\ge 0$ by \eqref{1.3} and the Cauchy-Schwarz inequality.
\end{proof}

In addition to the reproducing kernels in \eqref{2.2}, we also use their weighted versions in the proofs below:
\begin{equation*}
\tilde{K}_{n}^{\left( k,\ell \right) }\left( x,y\right) =\mu ^{\prime}\left( x\right) ^{1/2}\mu ^{\prime }\left( y\right)^{1/2}\sum_{j=0}^{n-1}p_{j}^{(k)}\left( x\right) p_{j}^{\left( \ell \right)}\left( y\right).
\end{equation*}

\begin{lemma} \label{lem3.2}
Let $\mu$ be a measure with compact support and with infinitely many points in its support. Let $O$ be an open set in which $\mu$ is absolutely continuous, and such that for some $C>1$ \eqref{2.4} holds. Then given any compact subinterval $[a,b]$ of $O$, we have
\begin{equation} \label{3.6}
\frac{1}{n} \E\left[ N_{n}\left([a,b]\right) \right] =\frac{1+o(1)}{\sqrt{3}}%
\int_{a}^{b}\frac{1}{n}K_{n+1}(x,x)\,d\mu(x).
\end{equation}
\end{lemma}

\begin{proof}
First note that the hypothesis that $\mu ^{\prime }\geq C^{-1}$ in $O$ gives \cite[Theorem 3.3, p. 104]{Fr}
\begin{equation*}
C_{1}=\sup_{n\geq 1}\sup_{x\in [a,b]}\frac{1}{n}K_{n+1}\left( x,x\right) <\infty.
\end{equation*}
Next, we use Corollary 1.4 in \cite[p. 224]{Lu2}. It gives for all $j,k\geq 0,$
\begin{equation} \label{3.7}
\lim_{n\rightarrow \infty }\int_a^b\left|\frac{\tilde{K}_{n+1}^{(j,k)}(x,x)}{\tilde{K}_{n+1}(x,x)^{j+k+1}} - \pi^{j+k}\tau _{j,k}\right|\,dx = 0.
\end{equation}
Here
\begin{equation*}
\tau _{j,k}=\left\{
\begin{tabular}{ll}
$0,$ & $j+k$ odd \\
$(-1)^{(j-k)/2}\frac{1}{j+k+1},$ & $j+k$ even.
\end{tabular}
\right.
\end{equation*}
Applying \eqref{1.2} in a modified form, we obtain that
\begin{equation} \label{3.8}
\frac{1}{n} \E\left[N_{n}\left([a,b]\right)\right] = \frac{1}{\pi }\int_{a}^{b} \sqrt{\frac{\tilde{K}_{n+1}^{\left( 1,1\right) }\left( x,x\right) }{\tilde{K}_{n+1}\left( x,x\right) ^{3}}-\left( \frac{\tilde{K}_{n+1}^{\left(0,1\right) }\left(x,x\right)}{\tilde{K}_{n+1}\left( x,x\right) ^{2}}\right) ^{2}}\frac{1}{n}\tilde{K}_{n+1}\left( x,x\right) dx.
\end{equation}
Since $\frac{1}{n}\tilde{K}_{n+1}\left( x,x\right)$ is bounded uniformly in $n$ and in $x\in[a,b]$, we can use \eqref{3.7} above to obtain
\begin{eqnarray*}
\frac{1}{n}E\left[ N_{n}([a,b])\right]  &=&\frac{1}{\pi } \int_a^b \left(\sqrt{\pi ^{2}\tau _{1,1}-\left( \pi \tau _{0,1}\right) ^{2}} + o\left( 1\right) \right) \frac{1}{n} \tilde{K}_{n+1}(x,x)\,dx \\
&=&\frac{1+o(1)}{\sqrt{3}} \int_a^b \frac{1}{n} \tilde{K}_{n+1}(x,x)\,dx.
\end{eqnarray*}
\end{proof}

\begin{proof}[Proof of Theorem \ref{thm2.2}]
Note that since $\mu ^{\prime }>0$ a.e. in $[a,b] $, this interval is contained in supp$\,\nu _{K}$. In \cite[p. 287, Theorem 1]{To1}, under weaker conditions, Totik proved that for a.e. $x\in [a,b],$
\begin{equation*}
\lim_{n\rightarrow \infty} \frac{1}{n} K_{n+1}(x,x) = \frac{d\nu_{K}}{d\mu}(x).
\end{equation*}
Since
\begin{equation*}
\lim_{n\rightarrow \infty} \frac{1}{n} \tilde{K}_{n+1}(x,x) = \frac{d\nu_{K}}{d\mu}(x)\,\mu'(x) = \nu_{K}^{\prime}(x),
\end{equation*}
the uniform boundedness of $\left\{ \frac{1}{n}\tilde{K}_{n+1}(x,x)\right\}_{n=1}^{\infty}$ and Lemma \ref{lem3.2} then give the result.
\end{proof}

\begin{proof}[Proof of Theorem \ref{thm2.3}]
We start with part (a). Given $r>0$, and $j,k\geq 0,$ with $\tau _{j,k}$ as above, it follows from \cite[p. 250, Proof of Corollary 1.4]{Lu2} that
\begin{eqnarray*}
&&\left\vert \frac{\tilde{K}_{n+1}^{\left( j,k\right) }\left( x,x\right) }{%
\tilde{K}_{n+1}\left( x,x\right) ^{j+k+1}}-\pi ^{j+k}\tau _{j,k}\right\vert
\\
&\leq &\frac{j!k!}{r^{j+k}}\sup_{\left\vert u\right\vert ,\left\vert
v\right\vert \leq r}\left\vert \frac{K_{n+1}\left( x+\frac{u}{\tilde{K}
_{n+1}\left( x,x\right) },x+\frac{v}{\tilde{K}_{n+1}\left( x,x\right) }
\right) }{K_{n+1}\left( x,x\right) }-\frac{\sin \left( \pi \left( u-v\right)
\right) }{\pi \left( u-v\right) }\right\vert .
\end{eqnarray*}%
Next, using that $\mu ^{\prime }>0$ a.e. in $[a,b]$, we have from \cite[p. 223, Theorem 1.1]{Lu2} that
\begin{equation*}
\text{meas}\left\{ x\in \left[ a,b\right] :\sup_{|u|,|v| \leq r} \left\vert \frac{K_{n+1}\left( x+\frac{u}{\tilde{K}_{n+1}(x,x)},x+\frac{v}{\tilde{K}_{n+1}(x,x)}\right)}{K_{n+1}\left( x,x\right)} - \frac{\sin(\pi(u-v))}{\pi(u-v)}\right\vert \geq \varepsilon
\right\} \rightarrow 0
\end{equation*}
as $n\rightarrow \infty$, for any given $\varepsilon,r>0$. Thus also
\begin{equation*}
\text{meas}\left\{ x\in \left[ a,b\right] :\left\vert \frac{\tilde{K}%
_{n+1}^{\left( j,k\right) }\left( x,x\right) }{\tilde{K}_{n+1}\left(
x,x\right) ^{j+k+1}}-\pi ^{j+k}\tau _{j,k}\right\vert \geq \varepsilon
\right\} \rightarrow 0
\end{equation*}
as $n\rightarrow \infty $. Now let $\varepsilon >0$, and for $n\geq 1$, let
\begin{equation*}
\mathcal{E}_{n}=\left\{ x\in \left[ a,b\right] :\sqrt{\frac{\tilde{K}%
_{n+1}^{\left( 1,1\right) }\left( x,x\right) }{\tilde{K}_{n+1}\left(
x,x\right) ^{3}}-\left( \frac{\tilde{K}_{n+1}^{\left( 0,1\right) }\left(
x,x\right) }{\tilde{K}_{n+1}\left( x,x\right) ^{2}}\right) ^{2}}\leq \sqrt{%
\pi ^{2}/3}-\varepsilon \right\} .
\end{equation*}
Then it follows that
\begin{equation*}
\text{meas}\left( \mathcal{E}_{n}\right) \rightarrow 0\text{ as }%
n\rightarrow \infty \text{.}
\end{equation*}%
Using \cite[p. 118, Thm. 2.1]{To2}, we have for a.e. $x\in[a,b]$ that
\begin{equation*}
\liminf_{n\rightarrow \infty }\frac{1}{n}\tilde{K}_{n+1}\left( x,x\right)
\geq \nu _{K}^{\prime }\left( x\right).
\end{equation*}
It then follows, that given $\varepsilon >0$,
\begin{equation*}
\mathcal{F}_{n}=\left\{ x\in \left[ a,b\right] :\frac{1}{n}\tilde{K}%
_{n+1}\left( x,x\right) \leq \nu _{K}^{\prime }(x) -\varepsilon \right\}
\end{equation*}
has
\begin{equation} \label{3.9}
\text{meas}\left( \mathcal{F}_{n}\right) \rightarrow 0\text{ as }%
n\rightarrow \infty.
\end{equation}%
Indeed, if we set
\begin{equation*}
f_{n}(x) = \min \left\{ \frac{1}{n}\tilde{K}_{n+1}(x,x) - \nu _{K}^{\prime }(x),0 \right\},
\end{equation*}
then by Totik's result,
\begin{equation*}
\lim_{n\rightarrow \infty} f_{n}(x) = 0 \text{ a.e. in }[a,b],
\end{equation*}%
while $f_{n}$ is bounded below by $-\nu _{K}^{\prime }$, so Lebesgue's Dominated Convergence Theorem gives
\begin{equation*}
0 = \lim_{n\rightarrow\infty} \int_{a}^{b} f_{n} \leq \liminf_{n\rightarrow\infty} (-\varepsilon) \text{meas}\left(\mathcal{F}_{n}\right).
\end{equation*}
Thus \eqref{3.9} holds. Then by \eqref{1.2}, \eqref{3.8} and the definitions of $\mathcal{E}_{n}$ and $\mathcal{F}_{n},$ we have
\begin{align*}
\frac{1}{n}\E\left[N_{n}\left([a,b]\right)\right] &= \frac{1}{\pi }\int_{a}^{b} \sqrt{\frac{\tilde{K}_{n+1}^{\left( 1,1\right) }\left( x,x\right) }{\tilde{K}_{n+1}\left( x,x\right) ^{3}}-\left( \frac{\tilde{K}_{n+1}^{\left(0,1\right) }\left(x,x\right)}{\tilde{K}_{n+1}\left( x,x\right) ^{2}}\right) ^{2}}\frac{1}{n}\tilde{K}_{n+1}\left( x,x\right) dx \\ &\geq \frac{1}{\pi}\int_{[a,b] \setminus \left(\mathcal{E}_{n}\cup \mathcal{F}_{n}\right)} \left(\sqrt{\pi^{2}/3}-\varepsilon\right) \left( \nu _{K}^{\prime }\left( x\right) -\varepsilon \right) dx \\ &\rightarrow \frac{1}{\pi }\int_a^b \left(\sqrt{\pi^{2}/3}-\varepsilon\right) \left( \nu _{K}^{\prime }\left( x\right) -\varepsilon \right) dx \quad\text{as } n\rightarrow\infty.
\end{align*}
Now we can let $\varepsilon \rightarrow 0$.

We pass to the proof of part (b). Let $L\subset K$ be a regular compact set such that the restriction $\mu\vert_L$ of $\mu$ to $L$ is STU regular, and $L$ contains $[a,b]$ in its interior. By monotonicity of the reproducing kernel (Christoffel function), if $K_{n}\left(\mu\vert_L,\cdot,\cdot\right)$ denotes the reproducing kernel of the measure $\mu\vert_L$, then for a.e. $x\in[a,b]\subset L$, Totik's result \cite[p. 287, Theorem 1]{To1} gives
\begin{eqnarray*}
&&\limsup_{n\rightarrow\infty} \frac{1}{n}K_{n+1}(x,x) \mu^{\prime}(x)  \\
&\leq &\limsup_{n\rightarrow\infty} \frac{1}{n}K_{n+1}\left(\mu\vert_L,x,x\right) \mu^{\prime}(x) = \nu_{L}^{\prime}(x).
\end{eqnarray*}
Moreover, $\left\{ \frac{1}{n}K_{n+1}\left(\mu\vert_L,x,x\right) \mu^{\prime }(x)\right\}_{n=1}^{\infty}$ is uniformly bounded in $[a,b]$. Then Lemma \ref{lem3.2} implies that
\begin{equation*}
\limsup_{n\rightarrow \infty} \frac{1}{n}\E\left[N_{n}\left([a,b]\right)\right] \leq \frac{1}{\sqrt{3}} \int_{a}^{b} \nu _{L}^{\prime}(x)\,dx.
\end{equation*}
Finally, taking the inf over all $L$ gives the result.
\end{proof}

\begin{lemma} \label{lem3.3}
Let $\mu$ be an STU regular measure on the real line with compact support $K$, and let $\nu_K$ be the equilibrium measure of $K$. Suppose that the coefficients of random orthogonal polynomials \eqref{2.1} are complex i.i.d. random variables such that $\E[|\log|c_0||]<\infty$. If $E\subset\C$ is any compact set satisfying $\nu_K(\partial E)=0,$ then
\begin{equation} \label{3.10}
\lim_{n\rightarrow\infty} \frac{1}{n}\E\left[N_n(E)\right] = \nu_{K}(E).
\end{equation}
\end{lemma}
\begin{proof}
Consider the normalized counting measure $\tau_n=\frac{1}{n}\sum_{k=1}^n \delta_{z_k}$ for a polynomial \eqref{2.1}, where $\{z_k\}_{k=1}^n$ are the zeros of that polynomial, and $\delta_z$ denotes the unit point mass at $z$. Theorem 2.2 of \cite{Pr} implies that measures $\tau_n$ converge weakly to $\nu_K$ with probability one. Since $\nu_K(\partial E)=0,$ we obtain that $\tau_n\vert_E$ converges weakly to $\nu_K\vert_E$ with probability one by Theorem $0.5^\prime$ of \cite{La} and Theorem 2.1 of \cite{Bi}. In particular, we have that the random variables $\tau_n(E)\rightarrow\nu_K(E)$ a.s. Hence this convergence holds in $L^p$ sense by the Dominated Convergence Theorem, as $\tau_n(E)$ are uniformly bounded by 1, see Chapter 5 of \cite{Gut}. It follows that
\[
\lim_{n\to\infty} \E[|\tau_n(E) - \nu_K(E)|] = 0
\]
for any compact set $E$ such that $\nu_K(\partial E)=0,$ and
\[
\left|\E[\tau_n(E) - \nu_K(E)]\right| \le \E[|\tau_n(E) - \nu_K(E)|] \to 0 \quad\text{as } n\to\infty.
\]
But $\E[\tau_n(E)]=\E[N_n(E)]/n$ and $\E[\nu_K(E)]=\nu_K(E),$ which immediately gives \eqref{3.10}.
\end{proof}

\begin{proof}[Proof of Theorem \ref{thm2.1}]
Given any $\eps>0$, we find a closed set $S$ satisfying the assumptions, and obtain from Theorem \ref{thm2.2} that
\[
\lim_{n\rightarrow\infty} \frac{1}{n}\E\left[N_n\left([a,b]\right)\right] = \frac{1}{\sqrt{3}} \nu_{K}([a,b])
\]
for any interval $[a,b]\subset K^\circ\setminus S$, where $K^\circ$ is the interior of $K$. Note that both $\E\left[N_n\left(H\right)\right]$ and $\nu_{K}(H)$ are additive functions of the set $H$. Moreover, they both vanish when $H$ is a single point by \eqref{3.10}, because $\nu_K$ is absolutely continuous with respect to Lebesgue measure on $K$, see \cite[Lemma 4.4.1, p. 117]{ST}. Hence \eqref{3.10} gives that
\[
\lim_{n\rightarrow\infty} \frac{1}{n}\E\left[N_n\left(\R\setminus S\right)\right] = \frac{1}{\sqrt{3}} \nu_{K}(\R\setminus S).
\]
We can find finitely many open intervals $I_k\subset\R,\ k=1,\ldots,m,$ covering $S$, with total length $\sum_{k=1}^m |I_k| < 2\eps.$ Let $R_k=\{x+iy:x\in I_k,\ |y|<1\},\ k=1,\ldots,m$, so that for $R=\cup_{k=1}^m R_k$ we have $S\subset R$ and $\nu_K(\partial R)=0$. Applying Lemma \ref{lem3.3} again, we obtain that
\[
\limsup_{n\rightarrow\infty} \frac{1}{n}\E\left[N_n\left(S\right)\right] \le \limsup_{n\rightarrow\infty} \frac{1}{n}\E\left[N_n\left(\overline R\right)\right] = \nu_K(\overline R \cap \R) = \nu_K\left(\cup_{k=1}^m \overline{I_k}\right),
\]
Absolute continuity of $\nu_K$ with respect to $dx$ implies that the last term in the above estimate tends to 0 as $\eps\to 0$. Thus \eqref{2.3} follows.
\end{proof}

\bibliographystyle{amsplain}

\end{document}